\documentclass[11pt,a4paper,reqno]{amsart}

\usepackage{a4wide}
\usepackage{graphicx}
\usepackage{latexsym}
\usepackage{epsfig}
\usepackage{amssymb}
\usepackage{amstext}
\usepackage{amsgen}
\usepackage{amsxtra}
\usepackage{amsgen}
\usepackage{amstext}
\usepackage{amsthm}
\usepackage{color}
\usepackage{enumitem}
\usepackage{subfigure}

\newtheorem{thm}{Theorem}[section]
\newtheorem{prop}[thm]{Proposition}
\newtheorem{lemma}[thm]{Lemma}

\newtheorem{rem}[thm]{Remark}

\theoremstyle{definition}

\newtheorem{definition}[thm]{Definition}

\newtheorem{scheme}{Scheme}

\numberwithin{equation}{section}

\newcommand{\R}{\mathbb{R}}

\newcommand{\Div}{\nabla \cdot}

\title{Mean field games with nonlinear mobilities in pedestrian dynamics}

\thanks{MTW acknowledges financial support of the Austrian Science Foundation FWF via the Hertha Firnberg Project T456-N23. MDF is supported by the FP7-People Marie Curie CIG (Career Integration Grant) Diffusive Partial Differential Equations with Nonlocal Interaction in Biology and Social Sciences (DifNonLoc), by the `Ramon y Cajal' sub-programme (MICINN-RYC) of the Spanish Ministry of Science and Innovation, Ref. RYC-2010-06412,
and by the by the Ministerio de Ciencia e Innovaci\'on, grant MTM2011-27739-C04-02. }

\begin{document}

\author{Martin Burger$^1$}
\address{$^1$ Institute for Computational and Applied Mathematics, University of M\"unster, Einsteinstrasse 62, 48149 M\"unster, Germany}
 \email{martin.burger@wwu.de}

\author{Marco Di Francesco$^2$}
\address{$^2$ Department of Mathematical Sciences, 4W, 1.14, University of Bath, Claverton Down, Bath (UK), BA2 7AY}
 \email{m.difrancesco@bath.ac.uk}

\author{Peter A. Markowich$^3$}
\address{$^3$ King Abdullah University of Science and Technology, Thuwal 23955-6900, Kingdom of Saudia Arabia}
 \email{P.A.Markowich@damtp.cam.ac.uk}

\author{Marie-Therese Wolfram$^4$}
\address{$^4$ Department of Mathematics, University of Vienna, Nordbergstrasse 15, 1090 Vienna Austria}
\email{marie-therese.wolfram@univie.ac.at}

\maketitle

\begin{abstract}
\noindent In this paper we present an optimal control approach modeling fast exit scenarios in pedestrian crowds.
In particular we consider the case of a large human crowd trying to exit a room as fast as possible.
The motion of every pedestrian is determined by minimizing a cost functional,
which depends on his/her position, velocity, exit time and the overall density of people.
This microscopic setup leads in the mean-field limit to a parabolic optimal control problem.
We discuss the modeling of the macroscopic optimal control approach and show how the optimal conditions relate
to Hughes model for pedestrian flow. Furthermore we provide results on the existence and uniqueness of minimizers
and illustrate the behavior of the model with various numerical results.\\
\end{abstract}

\section{Introduction}

\noindent Mathematical modeling of human crowd motion such as streetway pedestrian flows or the evacuation of large buildings is a topic of high
 practical relevance. The complex behavior of human crowds poses a significant challenge in the modeling and a majority of questions
still remains open. Consequently it receives increasing attention, and in particular in the last years strong development in the mathematical literature becomes visible. A variety of different approaches has been proposed and partly analyzed, which can roughly be grouped as follows:
\begin{enumerate}

\item \textit{Microscopic force-based models} such as the social force model (cf. \cite{HM1998, HFV2000, CWSS2011}) and their continuum limits.

\item \textit{Cellular automata} approaches (discrete random walks with exclusion principles \cite{BKSZ2001,BA2001}) and their continuum limits (cf. \cite{BMP2011}).

\item \textit{Fluid-dynamical and related macroscopic models} (cf. \cite{CGLM2012,Metal2012,CGR2010,ADM2011}), among which in particular the Hughes model and its variants received strong attention (cf. \cite{H2002,DFMPW2011, ADF2012,GM2012}).
\item Microscopic \textit{optimal control} and \textit{game} approaches (cf. \cite{HB2004,BPSML2008})  and \textit{mean field game} models (cf. \cite{LW2011,D2010}), which can be considered as the appropriate continuum limit (cf. \cite{LL2007, GLL2011}).

\end{enumerate}

\noindent For evacuation scenarios an additional complication in the basic model arises, namely how the boundaries (i.e. doors, impermeable walls) and the goal of the pedestrians to leave the room as quickly as possible are modeled. A canonical approach is to use a potential force (or drift in random walk models) based on the distance function to the doors. This means that the potential $\phi$ solves some kind of Eikonal equation in the domain $\Omega$ with zero Dirichlet boundary conditions on the part $\Gamma_E \subset \partial \Omega$, which corresponds to the door. The appropriate statement of the Eikonal equation and its boundary conditions in a crowded situation is a delicate issue, which is mainly done in an ad-hoc fashion (even for detailed microscopic models) and has led to different directions. Some models simply use the distance function itself, i.e. the viscosity solution of
\begin{equation}
	|\nabla \phi| = 1 \qquad \text{in } \Omega,
\end{equation}
which corresponds to an overall optimization of the evacuation path independent of other pedestrians. This might be questionable in crowds due to limited visibility or global changes of the evacuation path to avoid jamming regions. Other models use variants of the Eikonal equation that incorporate the density $\rho$ of pedestrians. The celebrated Hughes model, cf. \cite{H2002}, uses an equation of the form
\begin{equation}
 f(\rho)	|\nabla \phi| = 1 \qquad \text{in } \Omega,
\end{equation}
where $f$ is a function introducing saturation effects such as $f(\rho) = \rho_{\max} - \rho$, where $\rho_{\max}$ is a maximal density.

\noindent In this paper we introduce a rather unifying mean-field game approach to the modeling of evacuation scenarios, which links several of the above mentioned models and provides further understanding of several issues. This mean field framework provides interpretations of different forms of the Eikonal equation and covers the modeling of `realistic' boundary conditions. We provide a detailed mathematical analysis and numerical simulations of our approach and discuss interesting special cases and limits, in particular its connection to Hughes' model.

\noindent This paper is organized as follows: In Section \ref{s:model} we present the modeling on the microscopic level and its mean-field limit. Furthermore
we discuss its relation to Hughes' model and the appropriate choice of boundary conditions. Section \ref{s:analysis} is devoted to the
analysis of the macroscopic optimal control approach. Finally we present a steepest descent approach to solve the parabolic
optimal control problem in Section \ref{s:numerics} and illustrate the behavior with several numerical examples.

\section{Mathematical Modeling}\label{s:model}

\noindent In the following we discuss a model paradigm based on the idea of fast exit. At the single particle level, this yields a classical or stochastic optimal control problem, which we reformulate as PDE-constrained optimization model for the particle density. From the optimality conditions we then obtain a version of the Eikonal equation as the adjoint problem.

\subsection{Fast Exit of Particles}

Let us start with a deterministic particle (of unit mass) trying to leave the domain $\Omega$ as fast as possible. Let $X = X(t)$ denote the particle trajectory and
$$ T_{exit}(X) = \sup\{ t > 0 ~|~X(t) \in \Omega\}. $$
Then it makes sense to look for a weighted minimization of the exit time $T_{exit}$ and the kinetic energy, i.e.
\begin{equation}\label{eq:energy1}
	\frac{1}2 \int_0^{T_{exit}} |V(t)|^2~dt + \frac{\alpha}2 T_{exit}(X) \rightarrow \min_{X,V} ,
\end{equation}
subject to $\dot{X}(t) = V(t)$, $X(0) = X_0$, where $\alpha > 0$ encodes the weighting of the fast exit.

\noindent Introducing the Dirac measure $\mu = \delta_{X(t)},$  and a final time $T$ sufficiently large, we can rewrite the functional \eqref{eq:energy1} as
\begin{equation}\label{eq:energy2}
 I_T(\mu,v) = \frac{1}2 \int_0^T \int_\Omega |v(x,t)|^2 d\mu~dt + \frac{\alpha}2 \int_0^T \int_\Omega d\mu~dt ,
\end{equation}
subject to
\begin{equation}
	\partial_t \mu + \nabla \cdot (\mu v) = 0,
\end{equation}
and initial value $\mu|_{t=0} = \delta_{X_0}$.
Note that for $ T > T_{exit}$
$$ T_{exit}(X) = \int_0^T \int_\Omega ~d\delta_{X(t)} ~dt $$
and
$$  \int_0^T \int_\Omega |v(x,t)|^2 d\mu~dt  = \int_0^T \int_\Omega |v(x,t)|^2 ~d\delta_{X(t)}~dt =
\int_0^{T_{exit}} |v(X(t),t)|^2~dt, $$
which yields the equivalence between the particle formulation \eqref{eq:energy1} and the continuum version \eqref{eq:energy2} by the standard Lagrange-Euler transform of the velocity $V(t) = v(X(t),t)$, cf. \cite{HB2004,BPSML2008}.

\noindent Next we consider a stochastic particle, which moves according to the Langevin equation. Then we obtain
\begin{subequations}\label{e:stochopt}
\begin{equation}
	dX(t) = V(t)~dt + \sigma~dW(t), \label{Langevin}
\end{equation}
where $W$ a Wiener process and $\sigma$ the diffusivity. It is natural to consider the stochastic optimal control problem
\begin{equation}
\mathbb{E}_{X_0} \left[	\frac{1}2 \int_0^{T_{exit}} |V(t)|^2~dt + \frac{\alpha}2 T_{exit}(X)
\right] \rightarrow \min_{V}
\end{equation}
\end{subequations}
with the random variable $X$ determined by \eqref{Langevin} with initial value $X_0$. Reformulating \eqref{e:stochopt} for the distribution already reveals the structure of a mean-field game for the particle density. That is, writing $d\mu = \rho ~dx$ we obtain the minimization functional
\begin{subequations}
\begin{equation}
 I_T(\rho,v) = 	\frac{1}2 \int_0^T \int_\Omega \rho(x,t)~ |v(x,t)|^2  dx~dt + \frac{1}2 \int_0^T \int_\Omega \rho(x,t)~dx~dt ,
\end{equation}
subject to
\begin{equation}
	\partial_t \rho + \nabla \cdot (\rho v) = \frac{\sigma^2}2 \Delta \rho,~\rho(x,0) = \rho_0(x). \label{stateeqn}
\end{equation}
\end{subequations}

\subsection{Optimality, Eikonal Equations, and Mean-Field Games}

In the following we provide an explanatory derivation of the optimality conditions for the optimization problem
\begin{equation}
 	I_T(\rho,v) \rightarrow \hspace*{-1cm}\min_{(\rho,v) \text{ satisfying } \eqref{stateeqn}}\hspace*{-1cm}.
\end{equation}
We shall derive the optimality condition at a formal level, without providing rigorous details on the functional setting and on the boundary conditions. We shall be more precise later on a more general model, see subsection \ref{s:bc}.

\noindent We define the Lagrangian with dual variable $\phi = \phi(x,t)$ as
\begin{equation}
	L_T(\rho,v,\phi) = I_T(\rho,v) + \int_0^T \int_\Omega (
	\partial_t \rho + \nabla \cdot (\rho v) - \frac{\sigma^2}2 \Delta \rho) \phi~dx~dt.
\end{equation}
For the optimal solution we have, in addition to \eqref{stateeqn}, the following equations, i.e.
\begin{equation}
	0 = \partial_v L_T(\rho,v,\phi) = \rho v - \rho \nabla \phi,
\end{equation}
and
\begin{equation}
	0 = \partial_\rho L_T(\rho,v,\phi) = \frac{1}2 |v|^2 + \frac{\alpha}2 - \partial_t \phi - v \cdot \nabla \phi - \frac{\sigma^2}2 \Delta \phi,
\end{equation}
with the additional terminal condition $\phi(x,T) = 0$. Inserting $v=\nabla \phi$ we obtain the following system, which has the structure of a mean field game, i.e.
\begin{subequations}
\begin{align}
	\partial_t \rho + \nabla \cdot ( \rho \nabla \phi ) - \frac{\sigma^2}2 \Delta \rho &= 0 \\
	\partial_t \phi + \frac{1}2 |\nabla \phi|^2  + \frac{\sigma^2}2 \Delta \phi &= \frac{\alpha}2.
\end{align}
\end{subequations}

\noindent One observes the automatic emergence of a transient viscous Eikonal equation, which is to be solved backward in time with terminal condition. The connection to models using the distance function as a potential is quite apparent for a time interval $[0,S]$ with $S << T$ and $\sigma = 0$. Noticing that the Hamilton-Jacobi equation is solved backwards in time and that the backward time is large for $t \leq S$, we see that  the solution $\phi$ is mainly determined by the large-time asymptotics (cf. \cite{Ishii2008}) solving
\begin{equation}
	|\nabla \tilde \phi|^2 = c,
\end{equation}
for some constant $c$. Hence the potential becomes approximately a multiple of the distance function.


%

\subsection{Mean Field Games and Crowding}
We now turn our attention to a more general mean field game, respectively its optimal control formulation. We generalize the terms depending on the density for reasons to be explained in detail below to obtain
\begin{subequations}\label{e:optcon}
\begin{equation} \label{e:opt}
  I_T(\rho,v) = 	\frac{1}2 \int_0^T \int_\Omega F(\rho) |v(x,t)|^2 dx~dt + \frac{1}2 \int_0^T \int_\Omega E(\rho)~dx~dt ,
\end{equation}
subject to
\begin{equation} \label{e:constraint}
	\partial_t \rho + \nabla \cdot (G(\rho) v) = \frac{\sigma^2}2 \Delta \rho,
\end{equation}
\end{subequations}
and a given initial value $\rho(x,0) = \rho_0(x)$. The motivation for those terms is as follows:
\begin{itemize}

\item The function $G = G(\rho)$ corresponds to nonlinear mobilities, which are frequently used in crowding models, e.g. $G(\rho) =\rho(\rho_{\max}-\rho)$ derived from microscopic lattice exclusion processes (cf. \cite{BMP2011} for a pedestrian case). Nonlinear mobilities have been derived in different settings and used also in different applications of crowded motion, e.g. ion channels \cite{BDPS2010,BSW2012} or cell biology \cite{SHL2009, PH2002, DMB2012}. A typical form of $G$ is close to linear and increasing for small densities, while saturating and possibly decreasing to zero for larger densities.

\item The function $F=F(\rho)$ corresponds to transport costs created by large densities. In particular one may think of $F$ tending to infinity as $\rho$ tends towards a maximal density. As we shall see below however, changes in $F$ can also be related equivalently to changes in the mobility.

\item A nonlinear function $E=E(\rho)$ may model active avoidance of jams in the exit strategy, in particular by penalizing large density regions.
\end{itemize}

\noindent The optimality conditions of \eqref{e:optcon} can be (again formaly) derived via the Lagrange functional
\begin{equation}
	L_T(\rho,v,\phi) = I_T(\rho,v) + \int_0^T \int_\Omega (
	\partial_t \rho + \nabla \cdot (G(\rho)v) - \frac{\sigma^2}2 \Delta \rho) \phi~dx~dt.
\end{equation}
For the optimal solution we have
\begin{equation} \label{e:optv}
	0 = \partial_v L_T(\rho,v,\phi) = F(\rho) v - G(\rho) \nabla \phi,
\end{equation}
and
\begin{equation}
	0 = \partial_\rho L_T(\rho,v,\phi) = \frac{1}2 F'(\rho) |v|^2 + \frac{1}2 E'(\rho) - \partial_t \phi - G'(\rho) v \cdot \nabla \phi - \frac{\sigma^2}2 \Delta \phi,
\end{equation}
with terminal condition $\phi(x,T) = 0$. Inserting $v=\frac{G}F \nabla \phi$ we obtain the optimality system
\begin{subequations}\label{e:optsys}
\begin{align}
	\partial_t \rho + \nabla \cdot ( \frac{G(\rho)^2}{F(\rho)} \nabla \phi ) - \frac{\sigma^2}2 \Delta \rho &= 0 \\
	\partial_t \phi + \frac{1}2 (2 \frac{GG'}{F} - \frac{F'G^2}{F^2}) |\nabla \phi|^2  + \frac{\sigma^2}2 \Delta \phi &= \frac{1}2 E'(\rho),
\end{align}
\end{subequations}
again a mean field game.

\subsection{Momentum Formulation}

In order to reduce ambiguities in the definition of a velocity and also to simplify the constraint PDE, we use a momentum or flux-based formulation in the following, cf. \cite{ALS2006,BB2000}.
The flux density (or momentum) is given by $j = G(\rho) v$, hence the functional do be minimized in \eqref{e:optcon} becomes
\begin{subequations}\label{e:optconj}
\begin{equation} \label{e:optj}
  \widetilde{I}_T(\rho,j) = 	\frac{1}2 \int_0^T \int_\Omega \frac{|j|^2}{H(\rho)} dx~dt + \frac{1}2 \int_0^T \int_\Omega E(\rho)~dx~dt ,
\end{equation}
with
$H:=\frac{G^2}F$, subject to
\begin{equation} \label{e:constraintj}
	\partial_t \rho + \nabla \cdot j = \frac{\sigma^2}2 \Delta \rho.
\end{equation}
\end{subequations}
We see the redundancy of $F$ and $G$ in the above formulation, effectively only
$H:=\frac{G^2}F$ determines different cases. We mention that in order to obtain a rigorous formulation, we replace $\frac{|j|^2}{H(\rho)}$ in \eqref{e:optj} by
\begin{equation}\label{def_K}
	K(j,\rho):= \left\{ \begin{array}{ll} \frac{j^2}{H(\rho)} & \text{if } H(\rho) \neq 0 \\
	0 & \text{if } j=0, H(\rho) = 0 \\
	+ \infty & \text{if } j\neq 0 , H(\rho) = 0. \end{array}\right.
\end{equation}
The Lagrange functional for this problem is given by
\begin{equation}
	\widetilde{L}_T(\rho,v,\phi) = \widetilde{I}_T(\rho,j) + \int_0^T \int_\Omega (
	\partial_t \rho + \nabla \cdot j - \frac{\sigma^2}2 \Delta \rho) \phi~dx~dt,
\end{equation}
and the optimality conditions are
\begin{subequations}\label{e:optsysj}
\begin{align}
	\partial_t \rho + \nabla \cdot ( H(\rho) \nabla \phi ) - \frac{\sigma^2}2 \Delta \rho &= 0 \\
	\partial_t \phi + \frac{H'(\rho)}2  |\nabla \phi|^2  + \frac{\sigma^2}2 \Delta \phi &= \frac{1}2 E'(\rho).
\end{align}
\end{subequations}
We recall that the rigorous formulation of the problem in terms of the boundary conditions will be performed in subsection \ref{s:bc}.

\subsection{Relation to the Hughes Model}\label{s:hughes}
For vanishing viscosity $\sigma = 0$ the optimality system \eqref{e:optsys} respectively \eqref{e:optsysj} has a similar structure as Hughes model for pedestrian flow, see \cite{H2002}, which reads as
\begin{subequations}\label{e:hughes}
\begin{align}
\partial_t \rho - \Div(\rho f^2(\rho) \nabla \phi) &= 0\\
\lvert \nabla \phi \rvert &= \frac{1}{f(\rho)}. \label{e:hugheseikonal}
\end{align}
\end{subequations}
Hughes proposed that pedestrians seek the fastest way to the exit, but at the same time try to avoid congested areas. Let $\rho_{\max}$ denote the maximum density, then
the function $f(\rho)$ models how pedestrians change their direction and velocity due to the overall density. A common choice for example is $f(\rho) = \rho_{\max} - \rho$.\\
\noindent To obtain the connection with the Hughes model we choose $H(\rho) = \rho f(\rho)^2$, $E(\rho)=\alpha \rho$, then \eqref{e:optsysj}
reads as 
\begin{subequations}\label{e:timedepmodhughes}
\begin{align}
	\partial_t \rho + \nabla \cdot ( \rho~f(\rho)^2 \nabla \phi ) &= 0 \\
	\partial_t \phi + \frac{f(\rho)}2 (f(\rho) + 2 \rho f'(\rho)) |\nabla \phi|^2  &= \frac{\alpha}2.
\end{align}
\end{subequations}
For large $T$ we expect equilibration of $\phi$ backward in time. If we further consider $\sigma = 0$, then for time $t$ of order one the limiting model (very formal) becomes
\begin{subequations}\label{e:modhughes}
\begin{align}
	\partial_t \rho + \nabla \cdot ( \rho~f(\rho)^2 \nabla \phi ) &= 0 \\
	 (f(\rho) + 2 \rho f'(\rho)) |\nabla \phi|^2   &= \frac{c}{f(\rho)},
\end{align}
\end{subequations}
which is almost equivalent to Hughes model \eqref{e:hughes}, if we set $c=1$. Note that the difference is in the term
$2f'(\rho)$, which is however severe. For $f(\rho)=\rho_{\max}-\rho$ we have
$$ f(\rho) + 2 \rho f'(\rho) = \rho_{\max} - 3\rho, $$
thus for small densities the behavior is similar, but the singular point is $\rho=\frac{\rho_{\max}}3$. A (so far partial) attempt to a rigorous mathematical theory for the original Hughes model \eqref{e:hughes} has been provided in \cite{ADF2012}. Previous results \cite{DFMPW2011,CGLM2012} considered smoothed version of that model. El-Khatib et al. first considered a variant with a cost in the eikonal equation not related to $f(\rho)$, cf. \cite{goatin}, and finally Gomes and Saude provided
a priori estimates for \eqref{e:timedepmodhughes} in \cite{GS2013}.

\subsection{Boundary Conditions}\label{s:bc}

 We finally turn to the modeling of the boundary conditions, which we have neglected in the computations above.
To do so we consider \eqref{e:optconj} on a bounded domain $\Omega \subset \mathbb{R}^d$, $d=1,2$, corresponding to a room with one or several exits. We assume that the boundary $\partial \Omega$ is split into a Neumann part $\Gamma_N \subseteq \partial \Omega$, modeling walls, and the exits $\Gamma_E \subseteq \partial \Omega$, $\partial \Omega = \Gamma_N \cup \Gamma_E$ and $\Gamma_N \cap \Gamma_E = \emptyset$. We now formulate natural boundary conditions on the density, which will lead to adjoint boundary conditions for $\phi$.

\noindent On the Neumann boundary $\Gamma_N$ we clearly have no outflux, hence naturally
\begin{align*}
(-\frac{\sigma^2}{2} \nabla \rho + j) \cdot n  = 0,
\end{align*}
where $n$ denotes the unit outer normal vector.
The exit part is more difficult. Here the outflux depends on how fast people can leave the room.
If we denote the rate of passing the exit by $\beta$, then we have the outflow proportional to $\beta \rho$. Hence, we arrive at the Robin boundary condition
\begin{align}\label{e:robinbc}
(-\frac{\sigma^2}{2} \nabla \rho + j) \cdot n = \beta \rho .
\end{align}

\noindent Then boundary conditions for the adjoint variable $\phi$ can be calculated again from the optimality conditions. If we define the Lagrangian in this case we obtain:
\begin{align}
L_T(\rho,j,&\phi) = I_T(\rho,j) + \int_0^T \int_\Omega (\partial_t \rho + \nabla \cdot j - \frac{\sigma^2}2 \Delta \rho) \phi~dx~dt\nonumber\\
&= I_T(\rho,j) + \int_0^T \int_\Omega \rho(-\partial_t \phi - \frac{\sigma^2}2 \Delta \phi - j \cdot \nabla \phi) ~dx~dt\nonumber\\
&\phantom{= J_T(\rho,j)} + \int_0^T \int_{\partial \Gamma_E}(-\frac{\sigma^2}2 \nabla \rho \cdot n \phi + \frac{\sigma^2}2 \rho \nabla \phi \cdot n + j \cdot n \phi )~ ds~dt \nonumber\\
&\phantom{= J_T(\rho,j)} + \int_0^T \int_{\partial \Gamma_N} \frac{\sigma^2}2 \rho \nabla \phi \cdot n~ ds~dt \nonumber
\end{align}
\begin{align}
&= I_T(\rho,j) + \int_0^T \int_\Omega \rho(-\partial_t \phi - \frac{\sigma^2}2 \Delta \phi - j \cdot \nabla \phi) ~dx~dt\nonumber\\
&\phantom{= J_T(\rho,j)} + \int_0^T \int_{\partial \Gamma_E}(\beta \rho \phi + \frac{\sigma^2}2 \rho \nabla \phi \cdot n  ) ds~dt + \int_0^T \int_{\partial \Gamma_N} \frac{\sigma^2}2 \rho \nabla \phi \cdot n~ ds~dt.\label{eq:functional}
\end{align}
In this form we see that the optimality condition with respect to $\rho$, i.e. $\partial_\rho L_T = 0$, yields the following boundary conditions for the
adjoint variable $\phi$ in addition to the adjoint PDE:
\begin{align}
&\frac{\sigma^2}2 \nabla \phi \cdot n + \beta \phi = 0 \text{ on }  \Gamma_E \quad \text{ and } \quad \frac{\sigma^2}2 \nabla \phi \cdot n = 0 \text{ on }\Gamma_N. \label{e:adjointbc}
\end{align}

\noindent One observes that our model provides meaningful and easily interpretative boundary conditions for the density $\rho$ as well as for the adjoint variable $\phi$. A homogeneous Dirichlet boundary condition for $\phi$ on doors arises only in the asymptotic limit $\beta \rightarrow \infty$. In this case one also needs to specify a homogeneous Dirichlet boundary condition for $\rho$ for consistence. This is in contrast to previous models, which frequently used Dirichlet boundary conditions for $\phi$, but arbitrary Dirichlet values for $\rho$ or an outflux boundary condition as above with finite $\beta$.

\noindent In the case of Hughes model \eqref{e:hughes} commonly used boundary conditions are chosen as follows.
The potential $\phi$ is assumed to satisfy homogeneous Dirichlet boundary conditions, i.e. $\phi = 0$, while the density $\rho$ either
satisfies a Dirichlet boundary conditions of type $\rho = \rho_D$ or an outflux condition as in \eqref{e:robinbc}.
There are two possibilities to compare the behavior of solutions of \eqref{e:hughes} and \eqref{e:timedepmodhughes}:
\begin{itemize}
\item Either by prescribing a density  at the exit, i.e. $\rho = \rho_D$ and homogeneous Dirichlet boundary conditions for the adjoint, i.e. $\phi = 0$.  This, however, is not realistic in practical applications.
\item Or by choosing outflux boundary conditions for the density $\rho$, as in \eqref{e:robinbc} with a large proportionality constant $\beta$, and the corresponding no-flux boundary conditions \eqref{e:adjointbc} in the
for $\phi$. The latter choice for the potential $\phi$ is not compatible with \eqref{e:hughes}, but approximates the homogeneous Dirichlet boundary condition (typically chosen in Hughes model) for $\beta$ large.
\end{itemize}
Note that no flux boundary conditions for the density $\rho$, i.e. $\beta = 0$ in \eqref{e:robinbc} would result in a homogeneous Neumann boundary conditions \eqref{e:adjointbc} for $\phi$. Such boundary conditions are not compatible
with the Eikonal equation \eqref{e:hugheseikonal} in Hughes model and contradict the initial modeling assumptions.

\section{Analysis of the Optimal Control Model}\label{s:analysis}

\subsection{Existence of Minimizers}
In this section we discuss the existence and uniqueness of minimizers for the optimization problem \eqref{e:optcon}, respectively the
reformulation \eqref{e:optconj}. In order to avoid any ambiguity with the formulations we can simply set $$F=G=H,$$ which will remain as a standing assumption throughout the rest of this section.
Let $\Omega$ be a bounded domain in $\mathbb{R}^d,~d=1,2,$. We make the following basic assumption on $F$ and $E$:
\begin{enumerate}[label=(A\arabic{enumi})]
\item \label{a:reg} $F = F(\rho) \in C^1(\R)$, $F$ bounded, $E = E(\rho) \in C^1(\R)$ and $F(\rho) \geq 0$, $E(\rho) \geq 0$ for $\rho \in \Upsilon$.
\end{enumerate}
Existence of minimizers is guaranteed if
\begin{enumerate}[label=(A\arabic{enumi}), start=2]
\item \label{a:conv} $E = E(\rho)$ is convex.
\end{enumerate}
As recalled in subsection \ref{s:hughes}, models with maximal density are of great importance in the applied setting, therefore we shall consider the special class of mobilities $F$ which yields a density uniformly bounded by a maximal density. Let us fix the maximal density to be $\rho_{\max}>0$. Let us denote by $\Upsilon = [0,\rho_{\max}]$. Typically $\rho_{\max}=1$. In order to force the minimizers of the problem to satisfy $\rho\in \Upsilon$, we extend $F$ to $F(\rho)=0$ on $\rho \in \R\setminus\Upsilon$. This defines the next assumption on $F$, i.e.
\begin{enumerate}[label=(A\arabic{enumi}), start=3]
\item \label{a:crowd} $F(0)>0$ if $\rho\in \Upsilon$ and $F=0$ otherwise.
\end{enumerate}
Note that assumption \ref{a:crowd} is satisfied in particular by $F(\rho)=\rho f(\rho)$ in the Hughes model in subsection \ref{s:hughes}. To show uniqueness we need another additional assumption for $F$, namely
\begin{enumerate}[label=(A\arabic{enumi}), start=4]
\item \label{a:conc}$F = F(\rho)$ is concave.
\end{enumerate}
We consider the optimization problem on the set $V \times Q$, i.e. $I_T(\rho,v):  V \times Q \rightarrow \R$, where $V$ and $Q$ are defined as follows
\begin{align}
V &= L^2(0,T; H^1(\Omega))\cap H^1(0,T;H^{-1}(\Omega))\text{ and } Q = L^2(\Omega\times(0,T)).
\end{align}
Hence the optimization problem \eqref{e:optcon} reads as
\begin{align}\label{e:abstractmin}
  \min_{(\rho,v) \in V \times Q} I_T(\rho,v) \text{ such that } \partial_t \rho = \frac{\sigma^2}{2} \Delta \rho - \Div(F(\rho) v) ,
\end{align}
respectively in momentum formulation
\begin{align}\label{e:abstractminj}
  \min_{(\rho,j) \in V \times Q} \widetilde{I}_T(\rho,j) \text{ such that } \partial_t \rho = \frac{\sigma^2}{2} \Delta \rho - \Div(j) .
\end{align}

\noindent In order to have the differential constraint well defined in the prescribed functional setting, and to incorporate the Robin boundary conditions in the statement of the problem, we shall provide a more rigorous definition of the minimization problem below.

\noindent For the existence proof in the case of non-concave $F$, we introduce another formulation based on the rather nonphysical variable $$w=\sqrt{F(\rho)} v,$$ which has been used already in the case of linear $F$ in \cite{B2010}. Then the functional $I_T$ can be re-written as
\begin{equation}\label{e:functional_w}
    J(\rho,w)= \frac{1}2 \int_0^T\int_\Omega (|w|^2 +E(\rho))~dx~dt,
\end{equation}
and the optimization problem formally becomes
\begin{align}\label{e:abstractminw}
  \min_{(\rho,j) \in V \times Q} J(\rho,w)  \text{ such that } \partial_t \rho = \frac{\sigma^2}{2} \Delta \rho - \Div(\sqrt{F(\rho)} w).
\end{align}

\noindent In order to make the relation among the problems \eqref{e:abstractmin}, \eqref{e:abstractminj}, and \eqref{e:abstractminw} rigorous, we need to extend the domain of the velocity $v$ to
\begin{equation}
\tilde Q_\rho:= \{ v \text{ measurable} ~|~\sqrt{F(\rho)} v \in Q \}.
\end{equation}
Moreover, for given $\rho$ we define an extension mapping $w \in Q$ to $v \in \tilde Q_\rho$ via
\begin{equation}
R_\rho(w)(x):= \left\{ \begin{array}{ll} \frac{w(x)}{\sqrt{F(\rho(x))}} & \text{if } F(\rho(x)) \neq 0 \\ 0 & \text{else.}  \end{array} \right.
\end{equation}

\noindent With this notation, we easily get (we omit the details)
\begin{lemma}
Let \ref{a:reg} be satisfied and let $\sigma > 0$. Then the following relations hold:
\begin{itemize}
\item If $(\rho,v) \in V \times \tilde Q_\rho$ solves \eqref{e:abstractmin}, then $(\rho,\sqrt{F(\rho)}v) \in V \times Q$ solves \eqref{e:abstractminw}. If $(\rho,w) \in V \times Q$ solves \eqref{e:abstractminw}, then $(\rho,R_\rho(w)) \in V \times \tilde Q_\rho$ solves \eqref{e:abstractmin}.

\item If $(\rho,w) \in V \times \tilde Q_\rho$ solves \eqref{e:abstractminw}, then $(\rho,\sqrt{F(\rho)}w) \in V \times Q$ solves \eqref{e:abstractminj}. If $(\rho,j) \in V \times Q$ solves \eqref{e:abstractminj}, then $(\rho,R_\rho(j)) \in V \times \tilde Q_\rho$ solves \eqref{e:abstractminw}.
\end{itemize}
\end{lemma}

\noindent As a consequence of the above Lemma, we see that if one of the models has a minimizer with $F(\rho)$ uniformly bounded away from zero, it is in $V \times Q$ and minimizes all three models.
We are now ready to state the minimization problem \eqref{e:abstractmin} and its reformulation \eqref{e:abstractminw} rigorously, in a way to incorporate the Robin boundary conditions, and to take into account the proper functional setting.

\begin{definition}[Minimisation problem, 1st formulation]\label{def1}
Let $\rho_0\in L^2(\Omega)$. A pair $(\rho,v)\in V\times \tilde Q_\rho$ is a \emph{weak solution} to the minimization problem \eqref{e:abstractmin} with initial condition $\rho_0$, if $\rho(0)=\rho_0$ and
\begin{equation}\label{eq:weak_eqn}
\langle \partial_t \rho, \psi \rangle_{H^{-1},H^1} +  \int_\Omega( \frac{\sigma^2}{2} \nabla \rho- F(\rho) v) \cdot \nabla \psi ~dx  = - \int_{\Gamma_E} \beta \rho \psi ~ds,
\end{equation}
for all $\psi \in H^1(\Omega)$, and if
\begin{equation*}
    I_T(\rho,v)=\min\left\{ I_T( \bar\rho, \bar v),\ :\ (\bar \rho,\bar v)\in V\times Q,\ \ (\bar \rho,\bar v)\ \hbox{ satisfy }\ \ \eqref{eq:weak_eqn}\right\}.
\end{equation*}
\end{definition}

\noindent A similar re-formulation of \eqref{e:abstractminj} in terms of the momentum variable $j$ replacing $v$ can be also deduced.

\begin{definition}[Minimisation problem, 2nd formulation]\label{def2}
Let $\rho_0\in L^2(\Omega)$. A pair $(\rho,j)\in V\times \tilde Q_\rho$ is a \emph{weak solution} to the minimization problem \eqref{e:abstractminj} with initial condition $\rho_0$, if $\rho(0)=\rho_0$ and
\begin{equation}\label{eq:weak_eqn2}
\langle \partial_t \rho, \psi \rangle_{H^{-1},H^1} +  \int_\Omega( \frac{\sigma^2}{2} \nabla \rho- j) \cdot \nabla \psi ~dx  = - \int_{\Gamma_E} \beta \rho \psi ~ds,
\end{equation}
for all $\psi \in H^1(\Omega)$, and if
\begin{equation*}
    \widetilde{I}_T(\rho,j)=\min\left\{ \widetilde{I}_T( \bar\rho, \bar j),\ :\ (\bar \rho,\bar j)\in V\times Q,\ \ (\bar \rho,\bar j)\ \hbox{ satisfy }\ \ \eqref{eq:weak_eqn2}\right\},
\end{equation*}
where
\begin{equation*}
    \widetilde{I}_T(\rho,j) = \frac{1}2 \int_0^T \int_\Omega K(j,\rho) dx~dt + \frac{1}2 \int_0^T \int_\Omega E(\rho)~dx~dt,
\end{equation*}
and $K$ is defined as in \eqref{def_K}.
\end{definition}

\noindent Similarly, on the variables setting $(\rho,w)$ we get:

\begin{definition}[Minimisation problem, 3rd formulation]\label{def3}
Let $\rho_0\in L^2(\Omega)$. A pair $(\rho,w)\in V\times Q$ is a \emph{weak solution} to the minimization problem \eqref{e:abstractminw} with initial condition $\rho_0$, if $\rho(0)=\rho_0$ and
\begin{equation}\label{eq:weak_eqn3}
\langle \partial_t \rho, \psi \rangle_{H^{-1},H^1} +  \int_\Omega( \frac{\sigma^2}{2} \nabla \rho- \sqrt{F(\rho)} w) \cdot \nabla \psi ~dx  = - \int_{\Gamma_E} \beta \rho \psi ~ds,
\end{equation}
for all $\psi \in H^1(\Omega)$, and if
\begin{equation*}
    J_T(\rho,w)=\min\left\{ J_T( \rho, w),\ :\ (\bar \rho,\bar w)\in V\times Q,\ \ (\bar \rho,\bar w)\ \hbox{ satisfy }\ \ \eqref{eq:weak_eqn3}\right\}.
\end{equation*}
\end{definition}

We now prove our main existence result in the formulation using $w$, which allows also for non-concave $F$, and relies on a positive viscosity $\sigma$. First, we provide an a-priori estimate:

\begin{lemma} \label{rhowlemma}
Let $\rho_0\in L^2(\Omega)$. Let \ref{a:reg} and \ref{a:conv} be satisfied and let $\sigma > 0$, $\beta \geq 0$.
Let $w \in Q$ and let $\rho \in V$ be a weak solution of
\begin{equation}
\langle \partial_t \rho, \psi \rangle_{H^{-1},H^1} +  \int_\Omega( \frac{\sigma^2}{2} \nabla \rho- \sqrt{F(\rho)} w) \cdot \nabla \psi ~dx  = - \int_{\Gamma_E} \beta \rho \psi ~ds,
\end{equation}
for all $\psi \in H^1(\Omega)$. Then there exist constants $C_1, C_2 > 0$ depending on $F$, $\sigma$, $\Omega$ and $T$ only, such that
\begin{equation}
\Vert \rho \Vert_V \leq C_1 \Vert w \Vert_Q + C_2.
\end{equation}
\end{lemma}
\begin{proof}
We use the test function $\psi=\rho(t)$ and obtain
\begin{align*}
  \frac{1}{2} \frac{d}{dt}\int_\Omega \rho^2 dx = -\frac{\sigma^2}{2} \int\lvert \nabla \rho\rvert^2 dx + \int_{\Omega} \sqrt{F(\rho)} w \nabla \rho~dx - \int_{\Gamma_E} \beta \rho^2~ds.
\end{align*}
Since $F = F(\rho)$ is bounded, the second to last term is non positive, and $w(t) \in L^2(\Omega)$. This yields after integration in time
\begin{align*}
\frac 1 2 \lVert \rho (t) \rVert_{L^2}^2 +  \frac{\sigma^2}{4} \int_0^t \lVert \rho \rVert_{H^1}^2 d\tau \leq c_1 \int_0^t \lVert w \rVert_{L^2}^2~d\tau  + \frac{1}{2} \lVert \rho_0 \rVert_{L^2(\Omega)}^2,
\end{align*}
with a constant $c_1 \in \mathbb{R}^+$ depending only on $\sigma$ and $F$.
For the time derivative we deduce that
\begin{align*}
\lVert \partial_t \rho \rVert_{H^{-1}(\Omega)} = \sup_{\lVert \psi \rVert_{H^1}\leq 1} \lvert \langle \rho_t, \psi \rangle \rvert \leq \sup_{\lVert \psi \rVert_{H^1}\leq 1} \left[c_2 \lVert \rho \rVert_{H^1} \lVert \psi \rVert_{H^1} + c_3 \lVert w \rVert_{L^2} \lVert \nabla \psi \rVert_{L^2}\right],
\end{align*}
where the first term on the right-hand side includes the gradient terms in the domain as well as the boundary term via a trace theorem. Squaring and integrating in time finally yields the desired estimate for
$\lVert \rho \lVert_V$.
\end{proof}

\noindent The following technical Lemma is needed to ensure that assumption \ref{a:crowd} gives $\rho\in \Upsilon$. We provide the details for the sake of completeness, although the result is quite standard.
\begin{lemma}\label{lemma_crowd}
Assume $\rho$ and $w$ are as in the assumption of Lemma \ref{rhowlemma}, and assume further that $F$ satisfies \ref{a:crowd}. Then, $\rho(\cdot,t) \in \Upsilon$ for all $t\in (0,T]$ if $\rho_0(x) \in \Upsilon$.
\end{lemma}

\begin{proof}
Assume first $w\in C^1$, and let $\eta_\delta(\rho)$ be a $C^{1,1}$ regularization of the positive part $(\rho)_+=\max\{\rho,0\}$ as $\delta\searrow 0$, such that $\eta_\delta'(\rho), \eta''_\delta(\rho) \geq 0$, and $\eta''_\delta(\rho)=\frac{1}{\delta}\chi_{0<\rho<\delta}$. We can use the test function $\eta'(\rho-\rho_{\max})$ in Definition \ref{def3} to get
\begin{align}
    & \frac{d}{dt}\int_\Omega \eta_\delta(\rho-\rho_{\max})~dx = -\frac{\sigma^2}{2}\int_\Omega\eta''_\delta(\rho-\rho_{\max})|\nabla \rho|^2~ dx \nonumber\\
    & \ \ + \int_\Omega \sqrt{F(\rho)} w \eta''_\delta(\rho-\rho_{\max})\cdot\nabla \rho~ dx - \beta \int_{\Gamma_E}\rho \eta'_\delta(\rho-\rho_{max}) \nonumber\\
    & \ \leq \int_\Omega \sqrt{F(\rho)} w \eta''_\delta(\rho-\rho_{\max})\cdot\nabla \rho~ dx\nonumber\\
    & \ = - \int_\Omega \xi_\delta(\rho)\mathrm{div} w~ dx,\label{eq:estimate1}
\end{align}
where
\begin{equation*}
    \xi_\delta(\rho):=\int_{\rho_{\max}}^\rho\sqrt{F(\rho)}\eta''_\delta(z-\rho_{\max}) dz.
\end{equation*}
Now, it is clear that $\xi_\delta(\rho)=0$ if $\rho<\rho_{\max}$. Moreover, for $\rho_{\max}<\rho<\rho_{\max}+\delta$ we have
\begin{align*}
    & \xi_\delta(\rho)= \frac{1}{\delta}\int_{\rho_{\max}}^\rho \sqrt{F(z)} dz\leq C,
\end{align*}
since $F$ is uniformly bounded. Finally, in the case $\rho> \rho_{\max}+\delta$, we have
\begin{align*}
    & \xi_\delta(\rho)= \frac{1}{\delta}\int_{\rho_{\max}}^{\rho_{\max}+\delta} \sqrt{F(z)} dz\rightarrow \sqrt{F(\rho_{\max})} =0,
\end{align*}
as $\delta \searrow 0$. Combining all these assumptions, we get $ \xi_\delta(\rho)$ uniformly bounded and such that $\xi_\delta(\rho)\rightarrow 0$ almost everywhere on $\Omega\times [0,T]$. Since $w\in C^1$, we can integrate in time in \eqref{eq:estimate1} and send $\delta\searrow 0$ to get
\begin{equation*}
    \int_\Omega (\rho(x,t)-\rho_{\max})_+ dx \leq \int_\Omega (\rho_0(x)-\rho_{\max})_+ dx = 0,
\end{equation*}
which gives the assertion.
\noindent The general case $w\in Q$ can be recovered by a standard approximation argument, we omit the details.
\end{proof}

\begin{thm}[Existence in the general case]
Let $\rho_0\in L^2(\Omega)$. Let \ref{a:reg} and \ref{a:conv} be satisfied and let $\sigma > 0$. Then the variational problem \eqref{e:abstractminw} has at least a weak solution $(\rho,w) \in V \times Q$ with initial condition $\rho_0$ in the sense of Definition \ref{def2}. If in addition \ref{a:crowd} is satisfied, then $\rho\in \Upsilon$.
\end{thm}
\begin{proof}
We show the existence of a minimizer by the direct method. Let $(\rho_k,w_k) \in V \times Q$ be a minimizing sequence. Then we can assume without loss of generality that there exists a constant $c$, such that
$$
\int_0^T \int_\Omega (|w_k|^2 +E(\rho_k))~dx~dt \leq c.
$$
Clearly, $w_k$ is bounded in $L^2(\Omega \times (0,T))$ and hence has a weakly convergent sub-sequence with limit $\hat w$, again denoted by $w_k$. Now we use Lemma \ref{rhowlemma} to obtain boundedness of $\rho_k \in V$, from which we can extract another weakly convergent sub-sequence. From the Lemma of Aubin and Lions, cf. \cite{S1997}, we have a compact embedding of $\rho_k \in L^2(\Omega\times (0,T))$. Then the continuity and boundedness of $F$ implies that $\sqrt{F(\rho_k)}$ converges strongly in $L^2(\Omega\times (0,T))$ and the limit $\sqrt{F(\hat \rho)}$ is bounded. Hence, the product $\sqrt{F(\rho_k)}w_k$ converges weakly in $L^1(\Omega \times (0,T))$ and the limit $\sqrt{F(\hat \rho)} \hat w$ is in $L^2(\Omega\times (0,T))$, which allows to pass to the limit in the weak formulation of the constraint equation. Thus, $(\hat \rho, \hat w)$ is admissible, and the weak lower semicontinuity of the objective functional implies that it is indeed a minimizer. Note that the convexity of $E$ is crucial in this case, as it provides lower semi continuity of $\int_\Omega E(\rho_k) dx$.
\noindent Assume now that \ref{a:crowd} is satisfied. Then, due to Lemma \ref{lemma_crowd} for all $k$ we have $\rho_k\in \Upsilon$, which can be passed to the limit (possibly by extracting a further subsequence converging almost everywhere) to obtain the desired result.
\end{proof}


\noindent We now provide an alternative argument, which can be extended to the non-viscous case for zero-flux flux boundary conditions, i. e. for $\beta=0$. The argument below provides less regularity for the minimizer compared to the previous result, but the proof is much shorter.

\begin{thm}[Existence for concave mobility]\label{th:concave}
Let \ref{a:reg}, \ref{a:conv}, \ref{a:crowd}, and \ref{a:conc} be satisfied and let $\sigma > 0$. Then the variational problem \eqref{e:abstractminj} has at least one minimizer $(\rho,j) \in L^\infty(\Omega \times (0,T)) \times Q$ such that
$\rho(x) \in \Upsilon$ for almost every $x \in \Omega$. If $E$ is strictly convex the minimizer is unique.
\end{thm}
\begin{proof}
Let $(\rho_k,j_k)$ be a minimizing sequence. Since $F(\rho_k)$ is bounded, $j_k$ is uniformly bounded in $L^2(\Omega \times (0,T))$ and we can extract a weakly convergent sub-sequence with limit $\hat j$. By the boundedness of the objective functional, we have $F(\rho_k) > 0$ almost everywhere, which gives $\rho_k(x,t)\in \Upsilon$ for all $k$ and for all $(x,t)$. From the latter we can extract a sub sequence, again denoted by $\rho_k$ such that $\rho_k \rightharpoonup^* \hat \rho$ in $L^\infty(\Omega \times (0,T))$, respectively weakly in $L^p$ for $p \in (1,\infty)$. Since the constraint PDE is linear we can easily pass to the limit in a weak formulation and see that
$$ \partial_t \hat \rho + \nabla \cdot ( \hat j) = \frac{\sigma^2}2 \Delta \hat \rho, $$
holds in a weak sense. Moreover, the convexity of the objective functional guaranteed by \ref{a:conv}, \ref{a:crowd}, and \ref{a:conc} (recall that $1/F(\rho)$ is extended to $+\infty$ when $F(\rho)=0$)  implies weak semi-continuity an hence $(\hat \rho, \hat j)$ is a minimizer.
\noindent The uniqueness follows from a standard strict convexity argument.
\end{proof}

\begin{rem}
The result in Theorem \ref{th:concave} can be extended to the case $\sigma=0$ in case the boundary conditions are posed with $\beta=0$, i. e. no exits at the boundary. In this case, the boundary condition can be trivially passed to the limit $k\rightarrow +\infty$ above, whereas we are not able to close such argument in case $\beta>0$. This case requires a more refined analysis based on the theory of nonlinear conservation laws with boundary conditions, cf. \cite{bardos}.
\end{rem}

\begin{rem}
We mention that under the above assumptions one might expect Gamma-convergence of the optimal control problems as $\sigma \rightarrow 0$. The lower semicontinuity arguments can be used to verify the lower bound inequality, however it is so far unclear how the upper bound inequality can be verified. 
\end{rem}

\noindent Note that the convexity constraint \ref{a:conv} can be weakened, i.e. the function $E$ has to be weakly lower semicontinuous in suitable
function spaces to guarantee existence of minimizers. Since we only consider convex functions $E$, we did not state this more general result.

\subsection{Adjoint Equations and Optimality System}

In the following we investigate the existence of adjoints, i.e. the optimality condition with respect to the state variable $\rho$:
\begin{prop}
Let assumption \ref{a:reg} and \ref{a:conv} be satisfied and let $\rho$ be such that $H(\rho)\geq \gamma$ for some $\gamma > 0$. Then the adjoint equation of problem \eqref{e:abstractminj}, i.e.
\begin{subequations}\label{e:heat}
\begin{align}
\partial_t \phi + \frac{\sigma^2}{2} \Delta \phi & = \frac{1}{2} E'(\rho) - \frac{1}{2} \lvert j \rvert^2 \frac{F'}{F^2}\\
\phi(x,T) &= 0
\end{align}
\end{subequations}
with boundary conditions \eqref{e:adjointbc} and terminal data $\phi = 0$ at $t=T$ has a unique solution $\phi \in L^q(0,T; W^{1,q}(\Omega))$ with $q < \frac{N+2}{N+1}$.
\end{prop}

\noindent Note that the right hand side of \eqref{e:heat} is only in $L^1(\Omega \times (0,T))$. Existence and uniqueness of solutions $\phi$
follows from the results of Boccardo and Gallou{\"e}t, cf. \cite{BG1989}. However this result does not provide the necessary
regularity for $\phi$ to define the Lagrange functional $L_T$ properly.
%
%

\noindent Finally we investigate the full optimality system, where we eliminate $j=F(\rho) \nabla \phi$ and obtain
\begin{subequations}\label{e:mfgtype}
\begin{align}
\partial_t {\rho} &= \frac{\sigma^2}2 \Delta {\rho} - \nabla \cdot( F(\rho) \nabla \phi)  \label{e:mfgtype1}\\
\partial_t {\phi} &= -\frac{\sigma^2}2 \Delta {\phi} + \frac{1}{2} E'(\rho) - \frac{1}{2} F'(\rho) \lvert \nabla \phi \rvert^2  \label{e:mfgtype2}.
\end{align}
with boundary conditions
\begin{align}
(-\frac{\sigma^2}{2} \nabla \rho + F(\rho) \nabla \phi) \cdot n &=
\begin{cases}
0 \quad \text{on } \Gamma_N \\
\beta \rho \quad \text{on } \Gamma_E
\end{cases}
&\frac{\sigma^2}2 \nabla \phi \cdot n &=
\begin{cases}
0 \quad \text{on }\Gamma_N \\
-\beta \phi  \quad \text{on } \Gamma_E .
\end{cases}\label{e:boundary_opt}
\end{align}
\end{subequations}

\noindent To show uniqueness of the forward backward system \eqref{e:mfgtype} we follow an approach proposed by Lasry and Lions for mean field games, cf. \cite{LL2007}.

\begin{thm}[Uniqueness for the optimality system]
For a fixed initial condition $\rho_0\in L^2(\Omega)$, there exists a unique weak solution $$(\rho,\phi)\in L^2(0,T;H^1(\Omega))\times L^2(0,T;H^1(\Omega))$$ to the optimality system \eqref{e:mfgtype} with boundary conditions \eqref{e:boundary_opt}.
\end{thm}

\begin{proof}
Assume there exists two classical solutions $(\phi_1,\rho_1)$ and $(\phi_2, \rho_2)$ to the
optimality system \eqref{e:optsys}. We denote the difference between the solutions by $\bar{\rho} = \rho_1-\rho_2$ and $\bar{\phi} = \phi_1-\phi_2$, which satisfy
\begin{subequations}\label{e:diff}
\begin{align}
\partial_t \bar{\rho} &= \frac{\sigma^2}2 \Delta \bar{\rho} - \nabla \cdot( F(\rho_1) \nabla \phi_1 - F(\rho_2)\nabla \phi_2) \label{e:diff1}\\
\partial_t \bar{\phi} &= -\frac{\sigma^2}2 \Delta \bar{\phi} + \frac{1}{2} E'(\rho_1) - \frac{1}{2} E'(\rho_2) - \frac{1}{2} F'(\rho_1) \lvert \nabla \phi_1 \rvert^2 + \frac{1}{2} F'(\rho_2) \lvert \nabla \phi_2 \rvert^2 \label{e:diff2}.
\end{align}
\end{subequations}
Note that the differences $\bar{\rho}$ and $\bar{\phi}$ satisfy the following boundary conditions on $\Gamma_E$:
\begin{align}\label{e:bcbar}
(-\frac{\sigma^2}{2}\nabla \bar{\rho} + (F(\rho_1) \nabla \phi_1 - F(\rho_2) \nabla \phi_2)) \cdot n = \beta \bar{\rho} \text{ and } \frac{\sigma^2}{2} \nabla \bar{\phi} + \beta \bar{\phi} = 0.
\end{align}
Using \eqref{e:diff} we calculate (integrating by parts):
\begin{align*}
 & \frac{d}{dt}\int_{\Omega} \bar{\phi}\bar{\rho}~dx = \int_{\Omega} (\bar{\phi} \partial_t \bar{\rho} + \bar{\rho} \partial_t \bar{\phi} )~dx  \\
  &=\int_{\Omega} ( (F(\rho_1) \nabla \phi_1 -F(\rho_2) \nabla \phi_2) \nabla \bar{\phi})~dx+ \frac{1}{2}\int_{\Omega} ( (E'(\rho_1)-E'(\rho_2)) \bar \rho ) dx\\
&~~~ -\int_{\Omega} \frac{1}{2} (F'(\rho_1)\lvert \nabla \phi_1\rvert^2 - F'(\rho_2)\lvert \nabla \phi_2 \rvert^2) \bar \rho ~dx\\
&= \int_{\Omega} [(F(\rho_1) \nabla \phi_1 - F(\rho_2) \nabla \phi_2) \nabla \bar{\phi} - \frac{1}{2} (F'(\rho_1) \lvert\nabla \phi_1 \rvert^2 - F'(\rho_2) \lvert \nabla \phi_2 \rvert^2) \bar{\rho}]~dx\\
&~~~ + \frac{1}{2}\int_{\Omega} (E'(\rho_1)-E'(\rho_2) ) \bar \rho~ dx.
\end{align*}
Note that due to the boundary conditions \eqref{e:bcbar}, all boundary terms that result from the integration by parts in the previous calculation  vanish. Due to the concavity of $F$ we deduce that
\begin{align*}
&(F(\rho_1) \nabla \phi_1 - F(\rho_2) \nabla \phi_2) \nabla \bar{\phi} - \frac{1}{2} (F'(\rho_1) \lvert\nabla \phi_1 \rvert^2 - F'(\rho_2) \lvert \nabla \phi_2 \rvert^2) \bar{\rho} = \\
&= (F(\rho_1) - \frac{1}{2} F'(\rho_1) \bar{\rho}) \lvert \nabla \phi_1 \rvert^2 + (F(\rho_2) + \frac{1}{2} F'(\rho_2) \bar{\rho}) \lvert \nabla \phi_2 \rvert^2 - (F(\rho_1) + F(\rho_2)) \nabla \phi_1 \nabla \phi_2 \\
&\geq \frac{1}{2}(F(\rho_1) + F(\rho_2)) (\nabla \phi_1 - \nabla \phi_2)^2.
\end{align*}
Since $E$ is in $C^1(\Upsilon)$ and convex, $E'$ is monotone and we obtain
\begin{align*}
\frac{d}{dt}\int_{\Omega}& \bar{\phi}\bar{\rho}~dx \geq \int_{\Omega} \frac{1}{2}(F(\rho_1) + F(\rho_2)) (\nabla \phi_1 - \nabla \phi_2)^2~dx.
\end{align*}
We integrate over the interval $[0,T]$ and obtain (since $\bar{\rho}(x,0) = 0$ and $\bar{\phi}(x,T) = 0$) that:
\begin{align}\label{e:inequ}
\int_{\Omega} \frac{1}{2}(F(\rho_1) + F(\rho_2)) (\nabla \phi_1 - \nabla \phi_2)^2~dx \leq 0.
\end{align}
The function $F$ is positive, therefore \eqref{e:inequ} implies that
\begin{align*}
\nabla \phi_1 = \nabla \phi_2 \text{ in } \lbrace \rho_1 > 0 \rbrace \cup \lbrace \rho_2 > 0\rbrace.
\end{align*}
Then $\rho_1$ is a solution of $\partial_t \rho  = \frac{\sigma^2}2 \Delta \rho -\nabla \cdot (F(\rho) \nabla \phi_2) = 0$ and we conclude (by uniqueness) that $\rho_1 = \rho_2$.
\end{proof}

\section{Numerical simulations}\label{s:numerics}

\noindent In this section we present a steepest descent approach for the solution of the parabolic optimal control problem \eqref{e:optcon}. This method can be used for convex optimization problems and defines an iterative scheme to determine the optimal $\rho$ and $v$.

\noindent  The presented numerical simulations focus on the
different behavior of the classical model of Hughes \eqref{e:hughes} and the corresponding general mean field modification \eqref{e:timedepmodhughes}. This comparison requires the simulation of \eqref{e:hughes}, which is done using a finite volume scheme for the nonlinear conservation law and a fast sweeping method for the Eikonal equation, as presented by Di Francesco et al. in \cite{DFMPW2011}.

\noindent The steepest descent scheme for \eqref{e:optcon} can be written as:
\begin{scheme}
Let $\rho_0 = \rho_0(x)$, $\phi(x,t) = \phi_H(x)$, where $\phi_H$ is the solution of the Eikonal equation in the classical Hughes model \eqref{e:hughes} with $\rho(x,t) = \rho_0(x)$ for all $t \in [0,T]$ and $v = \frac{G(\rho)}{F(\rho)} \nabla \phi$, be the given initial data. Then the steepest descent scheme reads as:
\begin{enumerate}
\item Solve the nonlinear convection diffusion equation
\begin{subequations}\label{e:nonlin}
\begin{align}
&\partial_t \rho = \frac{\sigma^2}2 \Delta \rho - \nabla \cdot (F(\rho) v)\\
& (-\frac{\sigma^2}2 \nabla \rho + j)\cdot n =  \beta \rho \text{ on } \Gamma_E~ \text{ and }~(-\frac{\sigma^2}2 \nabla \rho + j)\cdot n = 0 \text{ on } \Gamma_N,
\end{align}
\end{subequations}
forward in time. System \eqref{e:nonlin} is solved implicitly in time. The resulting nonlinear equation $A(\rho,v) = 0$ is discretized using  a mixed hybrid discontinuous Galerkin (MHDG) method, cf. \cite{ES2010}, and solved using Newton's method.
\item Calculate the backward evolution of the adjoint variable $\phi = \phi(x,t)$ using the previously calculated density $\rho = \rho(x,t)$ in
\begin{align*}
&-\partial_t \phi - \frac{\sigma^2}2 \Delta \phi - G'(\rho) v \cdot \nabla \phi =  -\frac{1}{2}F'(\rho) \lvert v \rvert^2 - \frac{1}{2} E'(\rho),\\
&-\frac{\sigma^2}2 \nabla \phi \cdot n - \beta \phi = 0 \text{ on } \Gamma_E ~\text{ and }~-\frac{\sigma^2}2 \nabla \phi \cdot n = 0 \text{ on } \Gamma_N,
\end{align*}
with an implicit in time discretization and a MHDG method for the spatial discretization.
\item Update the velocity via $v = v - \tau (F(\rho) v - G(\rho) \nabla \phi)$, where $\tau$ is a suitably chosen step size.
\item Go to (1) until convergence of the function \eqref{e:opt}.
\end{enumerate}
\end{scheme}

\noindent Egger and Sch\"oberl presented a MHDG method for linear convection dominated problems, which can be adapted for the nonlinear
problems considered. We use the following basis functions for the discretization, i.e.
\begin{align*}
\rho, \phi, v \in P^0(T) \text{ and } \nabla \rho, \nabla \phi \in RT^0(T),
\end{align*}
where $P^0$ denotes piecewise constant basis function on the interval $T$ and $RT^0$ lowest order Raviart-Thomas basis functions.
The Newton iteration in \eqref{e:nonlin} is terminated, if $ \lVert A(\rho,v) \rVert_{L^2(\Omega)} \leq 10^{-6}$.

\noindent Throughout this section we consider the domain $\Omega = [-1,1]$ with exits located at $x=\pm 1$. The interval $\Omega$ is divided into
a set of equidistant intervals of length $h$. Furthermore the maximum density $\rho_{\max}$ is set to $1$.

\subsection{Comparison of the classical and the mean field type Hughes model}

We have seen in Section \ref{s:hughes} that the proposed mean field model has a similar structure as the classical model by Hughes.
Hence we want illustrate the behavior of both models with various experiments. In the first example we focus on the influence
of the boundary conditions as discussed in Section \ref{s:bc}. In the second example we illustrate the basic difference of the
mean field game approach and the classical model of Hughes.

\subsubsection{Behavior for different values of $\beta$}
We consider the time interval $t \in [0,3]$. In the classical Hughes model \eqref{e:hughes} the boundary conditions are set to
\begin{align*}
  \phi(\pm 1,t) = 0 \text{ and } (-\frac{\sigma^2}{2} \Delta \rho + \rho f(\rho) \nabla \phi) \cdot n = \beta \rho,
\end{align*}
with $f(\rho) = 1-\rho$. For the corresponding mean field model \eqref{e:timedepmodhughes} the functions $\rho$ and $\phi$ satisfy \eqref{e:robinbc} and \eqref{e:adjointbc} at $x = \pm 1$. In both cases we use the same diffusion
coefficient, i.e $\sigma = 0.1$. The spatial discretization is set to $h= 5 \times 10^{-2}$ in \eqref{e:hughes} and  $h=10^{-3}$ in \eqref{e:timedepmodhughes}, the time steps to $\Delta t=10^{-5}$ and $\Delta t=10^{-1}$ respectively. The different magnitudes of
the time stepping can be explained by the explicit in time discretization of \eqref{e:hughes} and the implicit time discretization
of \eqref{e:timedepmodhughes}. Figures \ref{f:hughesrho0333} and \ref{f:modhughesrho0333} show the evolution in time of the solutions of \eqref{e:hughes} and \eqref{e:timedepmodhughes} for different values of $\beta$. Although the models have a very similar structure, their
behavior is different. In the mean field model small congestions at the boundary are visible for $\beta=1$. Furthermore we do not observe
the immediate vacuum formation at $x=0$ as in Hughes model \eqref{e:hughes}. People rather tend to ``wait for a little while'' at the center
and then start to move at a higher speed. The expected equilibration of $\phi$ in \eqref{e:timedepmodhughes} is clearly visible in Figure \ref{f:modhughesrho0333} for all values of $\beta$.
\noindent
\begin{figure}
\begin{center}
\includegraphics[width=0.75 \textwidth]{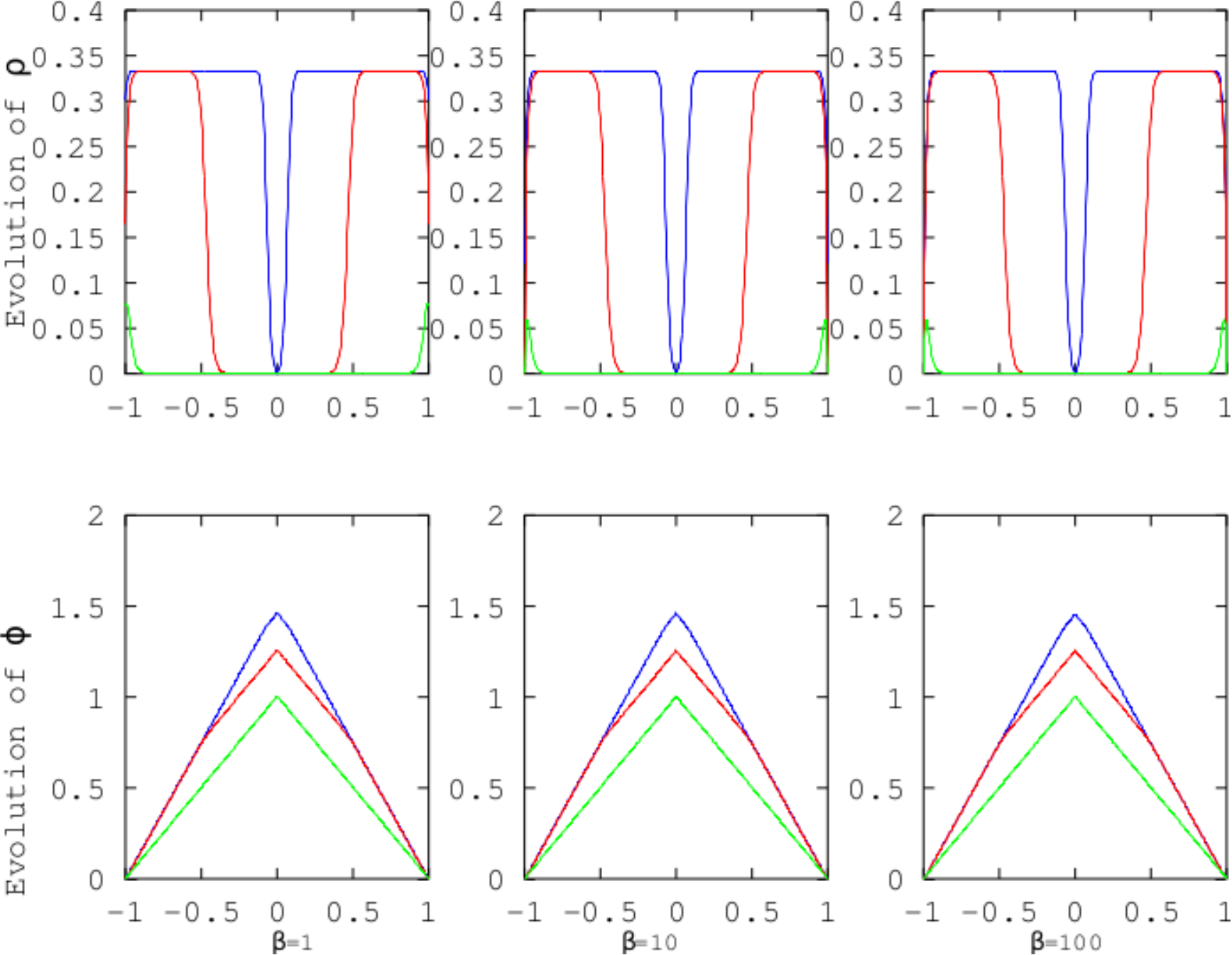}
\caption{Solution of the classical Hughes model \eqref{e:hughes} with initial datum $\rho_0(x) = \frac{1}{3}$ at times $t=0.1, 0.7, 1.5$ for different values of $\alpha$}\label{f:hughesrho0333}
\end{center}
\end{figure}

\begin{figure}
\begin{center}
\includegraphics[width=0.75 \textwidth]{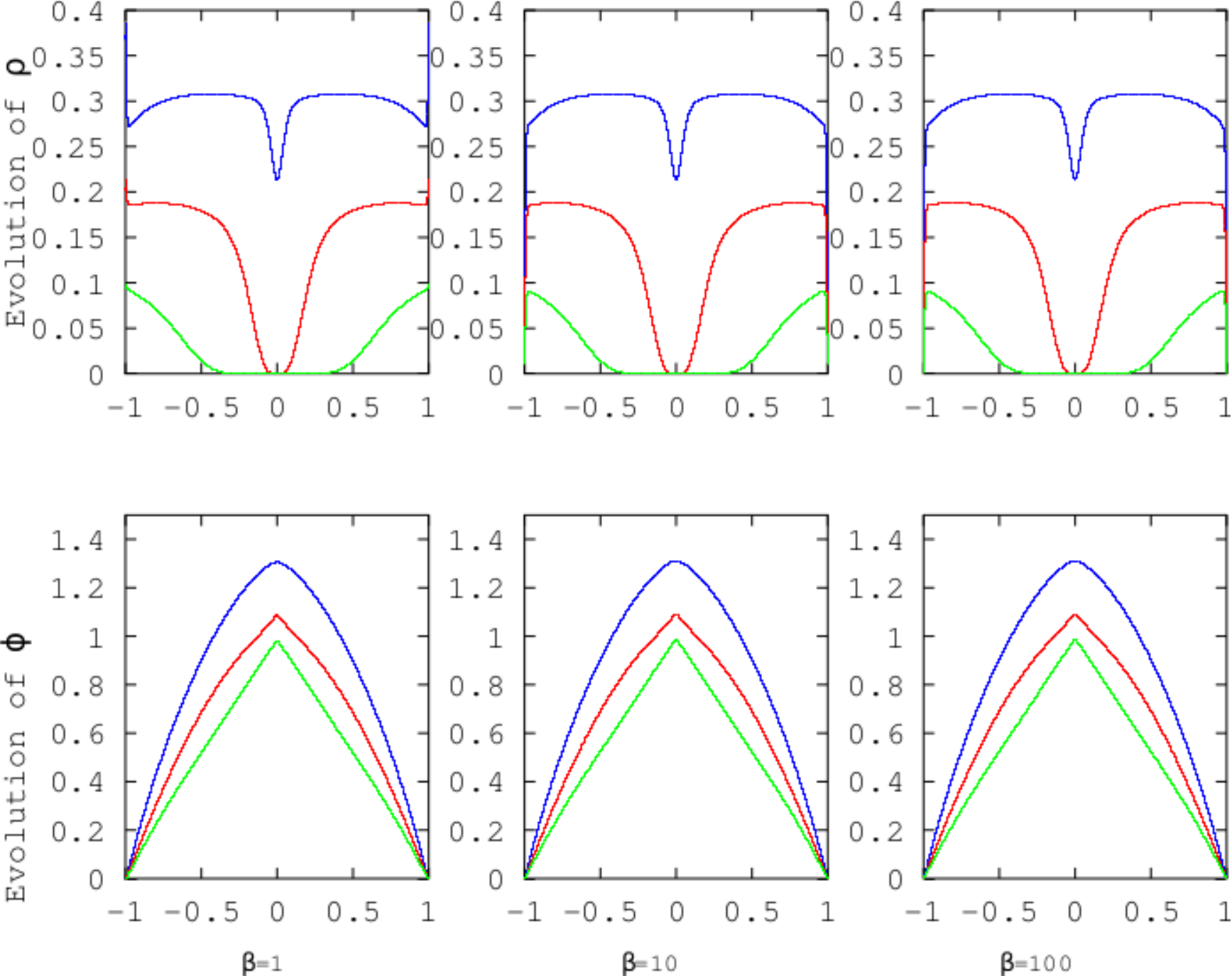}
\caption{Solution of the mean field type Hughes model \eqref{e:timedepmodhughes} with initial datum $\rho_0(x) = \frac{1}{3}$ at times $t=0.1, 0.7, 1.5$ for different values of $\alpha$}\label{f:modhughesrho0333}
\end{center}
\end{figure}

\subsubsection{Fast exit of several groups}

In this example we consider three groups, which want to leave the domain as fast as possible.
The particular initial datum is given by
\begin{align*}
\rho_0(x) &=
\begin{cases}
0.8 &\text{ if } -0.8 \leq x \leq -0.6 \\
0.6 &\text{ if } -0.3 \leq x \leq 0.3 \\
0.95 &\text{ if } 0.4 \leq x \leq 0.8 \\
0 &\text{ otherwise.}
\end{cases}
\end{align*}
A similar example was already considered in \cite{DFMPW2011}, where the simulations showed that the
a small part of the group located around $x=0.6$ initially splits to move to the exit at $x = -1$, but later
on turns around to take the closer exit at $x = 1$. This behavior can be explained by the fact that
the pedestrians in Hughes model \eqref{e:hughes}, adapt their velocity in every time step (depending
on the overall density at that time). Due to the initially high density of people located in front of exit
$x=1$, parts of the group start to move towards $x=-1$, but turn around when the density is decreasing
as more and more people exit.\\
We do not expect to observe this behavior for the modified Hughes model \eqref{e:timedepmodhughes}. One of
the underlying features of the proposed optimal control approach, is the fact that each pedestrian knows the distribution
of all other people at all times. Therefore he/she is anticipating the behavior of the group in the
future and will in this case rather wait than move to the more distant exit. This expected behavior can be observed
in Figure \ref{f:ex3}.
\begin{figure}
\begin{center}
\subfigure[Solution of the classical Hughes model \eqref{e:hughes}]{\includegraphics[width=0.45\textwidth]{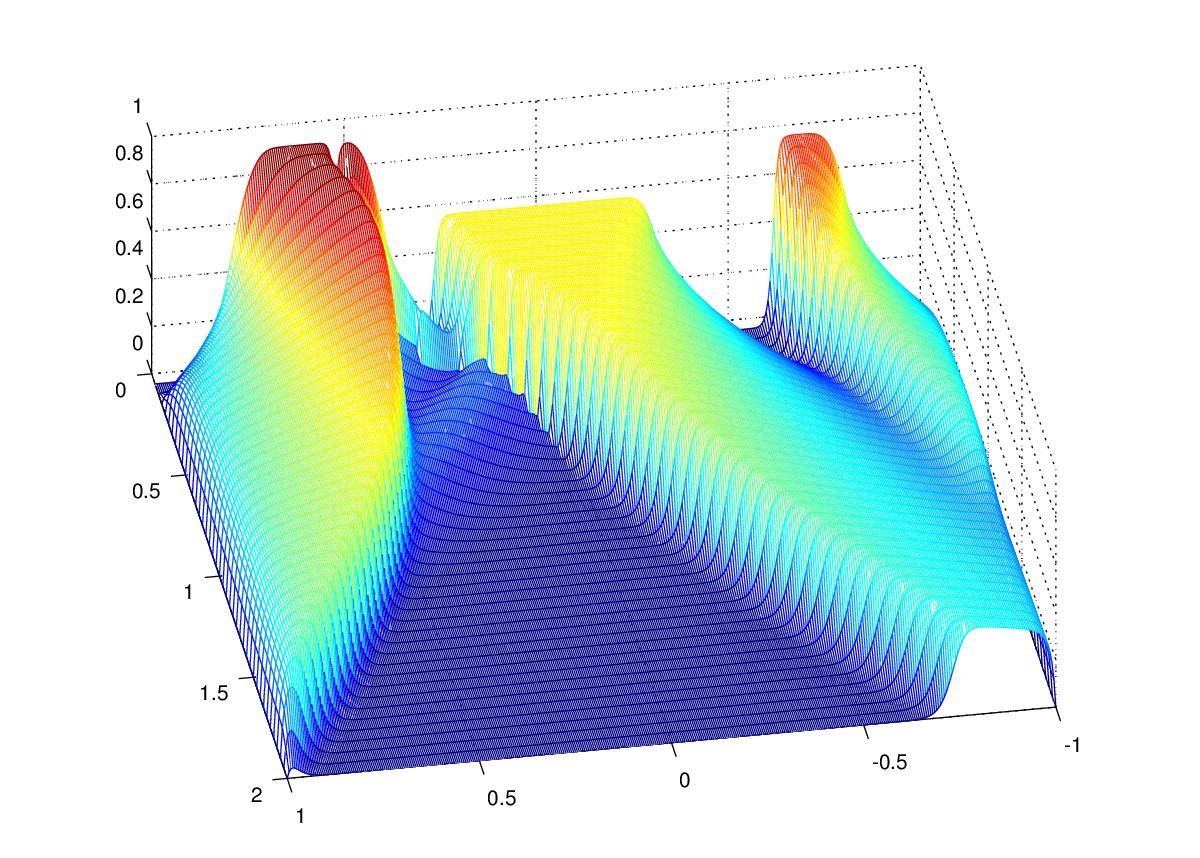}}
\vspace*{0.3cm}
\subfigure[Solution of the mean field optimal control approach \eqref{e:timedepmodhughes}]{\includegraphics[width=0.45\textwidth]{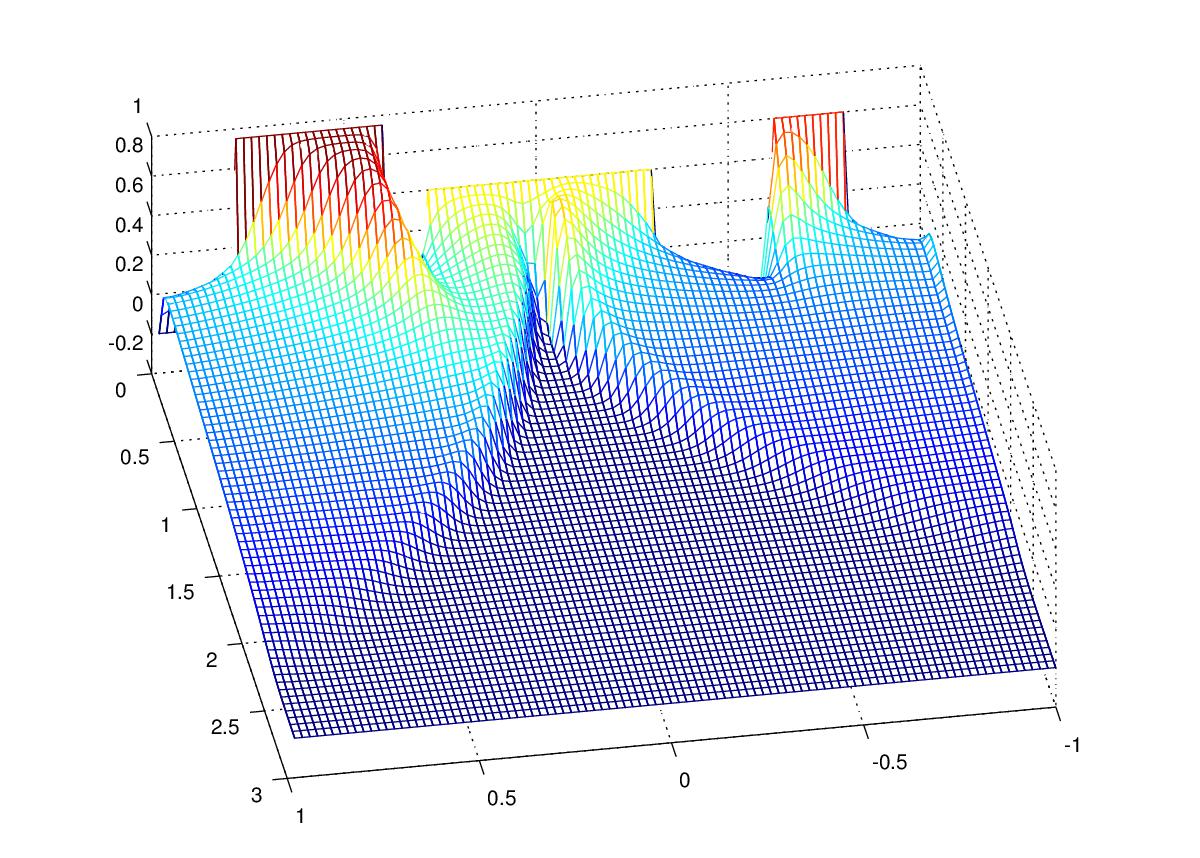}}
\caption{Fast exit scenario for three groups}\label{f:ex3}
\end{center}
\end{figure}
Here the mean field model \eqref{e:timedepmodhughes} was solved using the steepest descent approach detailed before. The parameters
were set to:
\begin{align*}
h = 10^{-3}, \Delta t= 5 \times 10^{-2}, T = 3 \text{ and } \beta = 1.
\end{align*}
Note that the optimal control formulation of \eqref{e:timedepmodhughes} is not a convex problem (the function $H(\rho) = \rho(1-\rho)^2$
is only concave for $\rho \in [0,\frac{2}{3}]$), but the
steepest descent approach converged without any problems. For Hughes model we chose the following parameters:
\begin{align*}
h = 5 \times 10^{-2}, \Delta t = 5 \times 10^{-5} \text{ and } \beta = 1.
\end{align*}

\subsection{Linear vs. nonlinear $E$}
In the final example we illustrate the behavior for different functions $E = E(\rho)$. In particular
we set
\begin{align*}
  F(\rho) = G(\rho) = \rho \quad  \text{ and } \quad E(\rho) &= \begin{cases} &3 \rho \\ &e^{3\rho} \end{cases}.
\end{align*}
The later choice of $E$ actively penalizes regions of high density and urges the group to get out
more efficiently. In this example the group is initially located at
\begin{align*}
\rho_0(x) = \begin{cases}
0.5 &\text{ if }  -0.25 \leq x \leq 0.4 \\
0 &\text{ otherwise}.
\end{cases}
\end{align*}
The parameters are given by $h = 10^{-3}, \Delta t=0.1, T=3$ and $ \beta = 1$. Figure \ref{f:ex4} shows the evolution of
the density $\rho$ for the different choices of $E$ in time. As expected we observe that the group exits much
faster and spreads out more evenly for $E = e^{3\rho}$ than in the linear case.
\begin{figure}
\begin{center}
\includegraphics[width=0.7\textwidth]{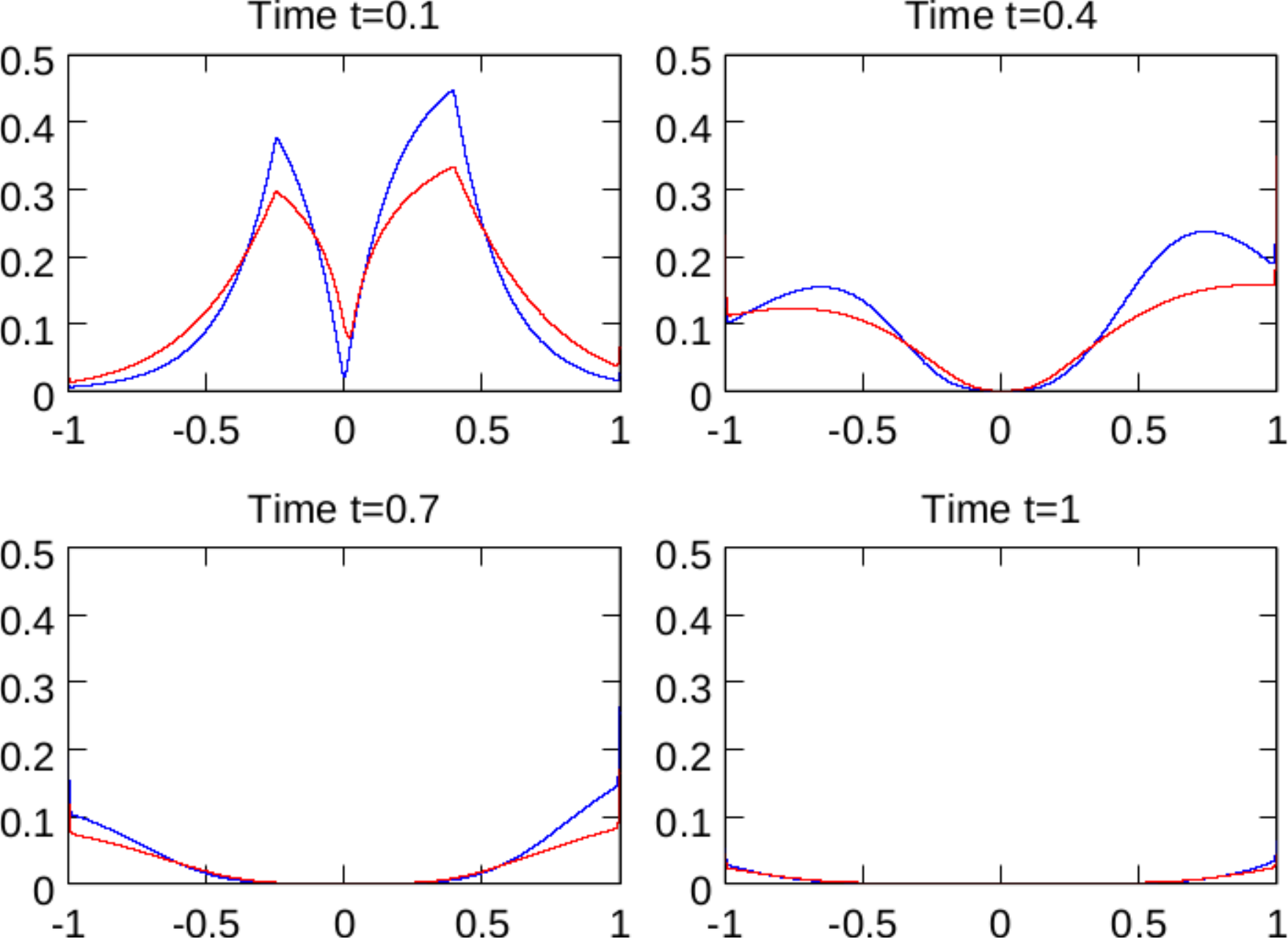}
\caption{Evolution of the density $\rho$ for different energies in time. The blue line corresponds to
  $E = 3 \rho$, the red one to $E=e^{3\rho}$.}\label{f:ex4}
\end{center}
\end{figure}

\section*{Conclusion}
In this paper we presented a novel mean field game approach for evacuation and fast exit situations in crowd motion.
We motivated the model on the microscopic level and discussed its generalization in the mean field limit. The optimality
system of the resulting parabolic optimal control problem, is a mean field game and establishes interesting links to well
known models, like Hughes model for pedestrian flow. Furthermore the proposed model gives new insights into the
mathematical modeling of boundary conditions for pedestrian crowds. We present first existence and uniqueness results
for the proposed optimal control approach and illustrate the behavior of solutions with various numerical simulations.

\noindent The general formulation of the proposed model poses interesting challenges for future research.
In general the derivation of mean-field limit with nonlinear mobilities is of significant importance and
raised a lot of interest in the scientific community, cf. e.g. \cite{CD2012}. In addition challenging questions and
problems arise in the mathematical analysis of the optimal control approach for convex as well as non-convex problems.
Another important direction of research will focus on more realistic modeling assumptions. We have seen in the second
example of Section \ref{s:numerics} that people anticipate the density of all others in time, which is questionable
in many situations. Therefore it would be interesting to study the
behavior of \eqref{e:optcon} with a temporal discount factor, as proposed in \cite{LW2011}. Finally the development of a
 numerical solver for 2D problems and its generalization to non-convex problems would allow
for more realistic simulations.

\bibliographystyle{siam}
\bibliography{hughes}

\end{document}